\documentclass{article}

\usepackage{amsmath}  
\usepackage{amssymb}  
\usepackage{amscd}    
\usepackage{enumerate}
\usepackage{comment}  
\usepackage{graphicx} 
\usepackage{multirow}            
\usepackage[dvipsnames]{xcolor} 
\usepackage[colorlinks=true,allbordercolors=white, citecolor=blue]{hyperref}

\usepackage{amsthm}   
\theoremstyle{plain}
\newtheorem{theorem}{Theorem}

\newtheorem{proposition}[theorem]{Proposition}
\newtheorem{corollary}[theorem]{Corollary}

\theoremstyle{definition}
\newtheorem{definition}[theorem]{Definition}

\newcommand{\dps}[1]{ \color{red} DPS: { #1} \color{black}}
\newcommand{\dfh}[1]{ \color{blue} DFH: { #1} \color{black}}

\newcommand{\pf}{\noindent{\bf Proof.} }

\newcommand{\F}{\mathbb{F}}
\newcommand{\Z}{\mathbb{Z}}

\newcommand{\R}{\mathbb{R}}

\newcommand{\si}{\mathcal{S}}
\newcommand{\alt}{\mathcal{A}}

\title{Essential quotients of the cohomology of products of symmetric and alternating groups}
\author{Dana Hunter and Dev Sinha \\ \small \href{mailto:dfhunter@fortlewis.edu}{dfhunter@fortlewis.edu} and \href{mailto:dps@uoregon.edu}{dps@uoregon.edu} }

\date{}

\begin{document}

\maketitle
{\let\thefootnote\relax\footnote{2020 Mathematics Subject Classifications 20J06, 12G05, 20B30}}

\begin{abstract}
    We calculate the quotient of the mod-two cohomology of
    a product of symmetric and alternating groups by the ideal
    of negligible elements, namely those that restrict trivially to any order-two
    subgroup. This quotient arises in the study of
    cohomological invariants of fields, as by work of Serre the ideal of negligible elements coincides with the ideal of elements that pull
    back trivially
    in Galois cohomology.
\end{abstract}

\section{Statements}

We focus exclusively on cohomology with $\Z/2$ coefficients, and algebras over $\Z/2$, letting $H^i(G)$ denote $H^i(G;\Z/2)$.

\begin{definition}\label{basicdef}
\begin{enumerate}
    \item 
An element $x \in H^*(G)$ is called {\bf negligible} if for
any field $k$ and continuous homomorphism from its Galois group $\phi: \Gamma_k \to G$ we have $\phi^*(x) = 0 \in H^*(\Gamma_k)$.

\item A class $x \in H^*(G)$ satisfies the {\bf order-two restriction property} if for any order-two subgroup $C \subset G$ the restriction of $x$ to $H^*(C)$ is zero.
\end{enumerate}

\end{definition}

Our work builds on results of Serre, namely Theorem 26.3 of \cite{Se03} and Theorem 7.4.1 of \cite{Serre26}, which relate these notions.

\begin{theorem}[Serre]\label{equiv}
    Let $G$ be a product of symmetric and alternating groups.
    An element of cohomology
of $G$ is negligible if and only if it satisfies the order-two
restriction property.  
\end{theorem}

\begin{definition}
    The negligible elements form an ideal, called the {\bf negligible ideal}. We call the quotient by 
this ideal the {\bf essential quotient}, denoted $H_{ess}^*(G)$.
\end{definition}

In this paper  we compute
these essential quotients for any product of symmetric and alternating groups.  To our knowledge these give
the first explicit full calculations beyond abelian groups \cite{B-FS94, GM24} and dihedral
groups \cite{Gherman23}. Structural theorems and low-degree calculations have been more of a focus, in particular of  Gherman and Merkurjev \cite{GM22, GM24}.

Our first  result is the following.  

\begin{theorem}\label{main}
Let $w_i$ denote the $i$th Stiefel-Whitney class of the standard
representation of $\si_{2n}$ acting through a permutation
of a basis, or its restriction to $\alt_{2n}$.
    \begin{itemize}
        \item $H^*_{ess}(\si_{2n})$ is generated by the classes $w_{i}$,  for $1 \leq i \leq n$.
        \item For $n \geq 3$, $H^*_{ess}(\alt_{2n} )$ is generated by the classes $w_{2j}$,  for $2 \leq 2j \leq n$.
        \item $H^*_{ess}(\alt_4 )$ is generated by $w_2$ and a class $c$ in degree three.

        \item If $G$ is a product of symmetric and alternating groups, 
        $H^*_{ess}(G)$ is generated by pull-backs of the above
        classes from constituent factors.
    \end{itemize}
\end{theorem}

The same results for odd indices follow, since with $\Z/2$ coefficients the cohomology
of $\si_{2n+1}$ and $\alt_{2n+1}$ are isomorphic to that of
$\si_{2n}$ and $\alt_{2n}$ respectively, and these isomorphisms
pass to essential cohomology.

 The relations between these classes follow from work of Serre 
within the context of cohomological invariants
of fields.  We build on and 
recast his results.

\begin{definition}
    Two homogeneous elements $a,b$ of a graded algebra are {\bf exchangeable} if $ab^n = a^mb$ for any positive $n,m$ for which these products are in the same degree.

    An {\bf exchange algebra} is a graded algebra for which there
    is a generating set
    with any pair of generators  exchangeable.
\end{definition}

In the appendix of \cite{Serre24}, Serre develops the case of 
exchange algebras with generators in the same degree -- that is,
those for which $a^2b = ab^2$. \

\begin{definition}
    A {\bf descendent} of a square-free monomial in an exchange algebra is a monomial with the same set of factors but where each factor can appear with any nonzero multiplicity.
\end{definition}

Descendents are uniquely determined
by the square-free monomial and the resulting degree,
which leads to the following.

\begin{proposition}\label{bound}
    The rank of an exchange algebra in each degree is
    bounded by the number of square-free monomials.
\end{proposition}

Before defining additive bases through collections of square
free monomials and their descendents, we  further develop exchange algebras.

 Over $\Z/2$,  if $x$ and $y$ are exchangeable then $(x+y)^n = x^n + y^n$ for any $n$.
So, if $x$ and $y$ are in the same degree and
can each be exchanged with $z$, 
then so can $(x+y)$.  Choosing any basis for a vector
space and forming the free commutative algebra on that 
basis modulo exchange relations will thus yield an algebra unique up to canonical isomorphism.
In fact, there is a free exchange algebra functor
from the category of graded vector spaces over $\Z/2$ to 
the category of graded exchange algebras over $\Z/2$.

We also have the following.

\begin{proposition}\label{subalgebra}
    A homogeneously generated subalgebra of an exchange algebra over $\Z/2$ is an exchange algebra.
\end{proposition}

\begin{proof}
    We check the exchange property for any two homogeneous elements, $x = \sum a_i$ and $y = \sum b_j$, where the $a_i$ and $b_j$
    are products of exchangeable generators:  
\begin{multline*}
x^n y^m = \left(\sum a_i\right)^n \left(\sum b_j\right)^m = \\
\left(\sum {a_i}^n\right)\left(\sum {b_j}^m\right) = \left(\sum {a_i}^m\right)\left(\sum {b_j}^n\right) = x^m y^n.
\end{multline*}

\end{proof}
\bigskip

Coproducts of exchange algebras are exchange algebras, and $\Z/2[t]$
is an exchange algebra, leading to the following.

\begin{theorem}\label{areexchange}
Let $G$ be a product
of symmetric and alternating groups.
Then $H^*_{ess}(G )$ is an exchange algebra.
\end{theorem}

\begin{proof}
        Let $\cal{C}$ be the set of order-two subgroups of $G$.
    Then by the order-two restriction criterion,
    the coproduct of restriction maps $H^*(G) \to \coprod_{C \in \cal{C}} H^*(C)$ passes to a map 
    $$H^*_{ess}(G) \to \coprod_{C \in \cal{C}} H^*(C) \cong \coprod_{C \in \cal{C}} \Z/2[t],$$ which is injective.  Because $H^*(G)$
    is homogeneously generated, so is $H^*_{ess}(G)$.  We apply
    Proposition~\ref{subalgebra}.
\end{proof}

Further relations between Stiefel-Whitney classes were  observed by Serre in the context of Witt rings.

\begin{definition}
    Call two natural numbers {\bf bitwise disjoint} if they share no places with 1’s
in their binary expansions.

Define the {\bf 2-part} of a finite set of natural
numbers to be the largest power of two which divides all
of them.
\end{definition}

\begin{theorem}\label{relations}
Let $G$ be a product
of symmetric and alternating groups.  In $H^*_{ess}(G )$
the following relations hold.
\begin{itemize}
    \item For any Stiefel-Whitney classes $w_i$ and $w_j$ pulled 
    back from the same factor,
    with $i$ and $j$ bitwise disjoint, $w_i \cdot w_j = w_{i+j}$.
    \item For classes $w_2$ and $c$ pulled back from any factor of $\alt_4$, we have $c^2 = {w_2}^3$.
\end{itemize}
Moreover, these relations and the exchange relations constitute
a complete set of relations.
\end{theorem}

The Stiefel-Whitney calculation follows from Equation 1.2.2 of \cite{Serre24}, which in turn cites work of Milnor \cite{miln70}.
We give an alternate proof below, using ideas of Serre in the group
cohomology context.
The relation on $\alt_4$ is elementary.  The result
that these relations are complete is new, though our methods
 also borrow from Serre.


An additive basis for $H^*_{ess}(G )$ follows from Theorems~\ref{main} and \ref{relations}, 
applying the fact that in an exchange algebra a monomial is completely determined by its constituent generators and total degree.  
We establish such a basis in the course of proving Theorem~\ref{relations} and record it here.

\begin{theorem}\label{basis}
Let $G$ be a product $\prod G_i$, where each $G_i$ is either
a symmetric or an alternating group. An additive basis for $H^*_{ess}(G)$ is given by the descendents of square-free
monomials $m$ which are products of:
\begin{itemize}
    \item a Stiefel-Whitney class $w_j$ with $0 \leq i \leq n$ pulled back from each factor $G_i = \si_{2n}$;
    \item a Stiefel-Whitney class $w_{2j}$ with $0 \leq 2j \leq m$ pulled back from each $G_i = \alt_{2m}$ with $m>2$;
    \item if $G_i = \alt_4$ then  the pull-back of either $1$, $w_2$, $c$ or $w_2 c$.

    \end{itemize}
The degree of a descendent differs from that of $m$ 
by  a multiple of the $2$-part 
of the degrees
of the constituent Stiefel-Whitney classes.
\end{theorem}
 
Thus, for example, 
the rank of $H^i_{ess}(\si_{2n})$ is less than
or equal to $n$ for all $i$.

 The structure of negligible ideals is also of interest.  In
 further work in progress, we show these are relatively simple to describe using
 the Hopf ring approach to the 
 cohomology of symmetric and alternating groups. While the subject of
cohomology of symmetric groups dates back to the 1950's,
generators and relations
for ring structure of $H^*(\si_n )$ were not
determined until relatively recent work of Feshbach  \cite{Fesh02}. Those presentations are complicated.
Later, Giusti, Salvatore and the second author found in \cite{GSS12} that by using
a second product, which is the analogue in cohomology 
of  external tensor product followed
by induction in representation theory, the presentation simplifies
substantially.  The resulting Hopf ring structure, incorporating 
both products, was crucial
in understanding cohomology of alternating groups in \cite{GiSi21}, which
previously had not even been determined additively.  

Our
first proofs of Theorem~\ref{main} used this framework.
They were shorter, in a sense, but relied on the full force of \cite{GSS12, GiSi21}.
The present proofs only employ the
time-honored method of restricting to the cohomology
of elementary abelian subgroups,
as well as the theory of characteristic classes.

We thank J-P. Serre for bringing these questions to our attention as well as 
for answering our questions about his work.  It has
been inspiring to communicate about active research mathematics
with him in his hundredth year (soon to be $101^{st}$).
We thank Eva Bayer-Fluckiger for careful readings of and many helpful comments 
about this work.

\section{Proofs}

Our strategy for proofs of Theorems~\ref{main}, \ref{relations} and
\ref{basis} is to make calculations of suitable
restrictions to the cohomology of elementary abelian
subgroups.  

\begin{definition}
An elementary abelian $2$-subgroup $E \subset G$ is
{\bf detecting} if any order-two subgroup of $G$ is conjugate
to a subgroup of $E$.  
\end{definition}

If $G$ is a product of symmetric and alternating groups
and $E \subseteq G$ is detecting, then $H^*_{ess}(G ) \to H^*_{ess}(E )$ is injective.
This is immediate from the order-two restriction property, since restricting to zero on all
order-two subgroups of $E$ will imply such for all
order-two subgroups of $G$.

We recall a standard starting point in group cohomology.
First, $G$  acts by conjugation on its classifying space, or equivalently its bar resolution, 
but the resulting action on cohomology
is trivial (\cite{AdMi94} II.2 Lemma 3.1).  Similarly, if $H \subset G$ is a subgroup, 
the normalizer $N_H$ of $H$  acts on the cohomology of
$H$.  This action can be nontrivial, though $H \subset N_H$
 acts trivially.  As is standard in group cohomology (but 
unfortunately is in conflict with usage elsewhere), 
we call the quotient
$N_H / H$ the Weyl group $W_H$. With any constant coeffients
restriction factors through Weyl invariants:
$$res: H^*(G) \to H^*(H)^{W_H} \subset H^*(H).$$

We now consider the first part of Theorem~\ref{main}, which
follows from a well-known calculation.

\begin{definition}
    Let $E_{sym}$ denote the subgroup generated by $(1 \; 2), (3 \; 4), \ldots, (2n - 1 \; 2n)$ in $\si_{2n}$. 
\end{definition}

\begin{proof}[{\bf Proof of first part of Theorem~\ref{main}}]
Because any order-two permutation is a product of
disjoint two-cycles, the subgroup $E_{sym}$ is detecting.
Its Weyl group as a subgroup
of $\si_{2n}$ is $\si_n$ (see \cite{AdMi94} VI.3),  
so the restriction map on cohomology
has image in the ring of symmetric
polynomials.  We show that it constitutes that
entire ring by showing that Stiefel-Whitney classes
restrict to elementary symmetric functions, a result
proven, for example, in \cite{MaMi79}.  For completeness we
give the argument here.

By definition,
Stiefel-Whitney classes are associated to the vector bundle  $V_{std}$ induced by the standard representation of $\si_{2n}$, 
 which acts on a basis $\{e_i\}$
by permutation of indices. 
  The pull-back of $V_{std}$ to $E_{sym}$ is $\bigoplus_i (\tau_i + \R)$, where $\tau_i$ denotes the sign representation of the $i$th factor of $C_2 \subset E_{sym}$ 
  and $\R$ denotes the trivial bundle.  Indeed, we can identify $\tau_i$ as the pullback of the span of 
$e_{2i -1} - e_{2i}$ and the $i$th trivial bundle as the pullback of
the span of 
$e_{2i -1} + e_{2i}$.  The Whitney sum formula now shows that $w_i$ maps to the $i$th symmetric polynomial, which
establishes the result.

We also deduce one case of second part of Theorem~\ref{main},
namely that of $\alt_{4n+2}$.  This immediately follows, 
since its
cohomology is the quotient of that of
$\si_{4n+2}$ modulo the Euler class associated to the two-sheeted
covering map which relates their classifying spaces (see Corollary VI.6.2 of \cite{AdMi94}).
This Euler class is also the first Stiefel-Whitney class.
Because the map in cohomology is surjective, so is
the resulting 
well-defined map on order-two essential quotients,
which implies that Stiefel Whitney classes generate
$H^*_{ess}(\alt_{4n+2} ).$
\end{proof}

Our argument to show Stiefel-Whitney classes generate
essential cohomology of $\alt_{4n}$ is similar in outline, though with
an added complication that we need to use relations
in the essential quotient of our chosen detecting subgroup.  We thus
recall the following result.

\begin{theorem}[Bayer-Fluckger and Serre]\label{elemabelian}
Let $E$ be an elementary abelian $2$-group.  Then $H^*_{ess}(E )$ is the free exchange algebra on $H^1(E )$. 
\end{theorem}

This theorem was stated in the context of cohomological
invariants in Proposition~7.2.4 of \cite{B-FS94}. We give a 
proof using  the order-two restriction property.

\begin{proof} 
The order-two subgroups of $E$ correspond to non-identity elements of $E$ and the restriction on $H^1(E ) \cong {\rm Hom}(E )$ is given by the restriction of homomorphisms. Thus, the  negligible ideal is the intersection
of kernels of all homomorphisms $\phi_S : H^*(E ) \cong \Z/2[x_1, \ldots, x_n] \to \Z/2[t]$, naturally indexed by the set of variables $S$ which
map to $t$.   Since ${x_i}^2 x_j + x_i {x_j}^2 = x_i \cdot x_j \cdot (x_i + x_j),$ it will be zero if either $x_i$ maps to zero, $x_j$ maps to zero, or if they both map to $t$.  Thus,  exchange relations among
the $x_i$ are in 
the kernels of all $\phi_S$.

To show the converse inclusion, 
we show that any homogeneous
polynomial $f$ in the kernel of all $\phi_S$ 
is equivalent to 
zero modulo the exchange relations.  
Modulo these, any monomial is equivalent to a monomial
in which all but one of the variables 
are raised to the first
power.  Thus all monomials which are a product of
positive powers of generators with labels in
$ S = \{ i_1, \ldots, i_k \} $ are equivalent to a multiple $m_S$ of
${x_{i_1}}^{p-k+1} x_{i_2} \cdots x_{i_k}$.  We show inductively
that this multiple must be zero.  If $S = \{ i \}$ 
then $\phi_{\{i\}}$  sends $f$ to $m_{\{i\}} t^p$, and so would have zero image only if $m_{\{i\}} = 0$.  If inductively we have 
shown that $m_T$ is zero for $T \subset S$ then the image of $f$ under $\phi_S$
is $m_S t^p$, showing $m_S = 0$.
\end{proof}


We now proceed to define a detecting subgroup, understand its
Weyl group, and then show that Stiefel-Whitney
classes exhaust the Weyl invariants.

\begin{definition}
    Let $V_2 \subset \alt_4$ be the subgroup given by the identity element and $(12)(34), (13)(24)$ and $(14)(23)$.  Define $E_{alt}$ as  ${V_2}^{\times n} = V_2 \times \cdots \times V_2 \subset \alt_{4n}$,  
 the subgroup given by the composite 
 with the standard inclusion of $\prod \alt_4 \subset \alt_{4n}$. 
\end{definition}

To understand the normalizer and Weyl groups of $E_{alt}$, we first consider
the symmetric group setting.  

\begin{definition}
    The action of $GL_2( \F_2)$ 
    on ${ \F_2}^2$, along with a fixed choice of isomorphism 
    between ${ \F_2}^2$ and $\{1, 2, 3, 4\}$,
    defines an embedding $GL_2( \F_2) \subset \si_4$.

     Embed  ${GL_2( \F_2)}^n$ in $\si_{4n}$ through the composite of the $n$-fold product of these embeddings  with the standard inclusion ${\si_4}^{n}$ $ \subset \si_{4n}$. 
     
    Embed
    $\si_n \subset \si_{4n}$ through order-preserving
    permutation of the subsets $\{4i+1, \cdots, 4i + 4\}$. 
    
    Define $W_n^{sym}$ to be the subgroup generated by these two subgroups.

\end{definition}

As stated in 
Section 3 of \cite{MaMi79} and
Section VI.4 of \cite{AdMi94},
    the Weyl group of ${V_2}^{\times n} \subset \si_{4n}$ is isomorphic to $W_n^{sym}$, which moreover is isomorphic to $\si_n \wr GL_2( \F_2)$.
From this the corresponding result for alternating groups  is immediate.

\begin{definition}
    Define $W_n$ to be the intersection of $W_n^{sym}$ with $\alt_{4n}$,
    or equivalently the kernel of the sign homomorphism restricted to
    $W_n^{sym}$.
\end{definition}

 By construction, $W_n$ normalizes $E_{alt}$ in $\si_{4n}$ and 
thus $\alt_{4n}$.  Moreover, this normalizer can have at most index two in
$W_n^{sym}$, as $W_n$ does.  We thus have the following.

\begin{corollary}
    The Weyl group of $E_{alt}$ in $\alt_{4n}$ is $W_n$.
\end{corollary}

Recall  that $GL_2( \F_2) \cong C_3 \rtimes C_2$ and $GL_2( \F_2) \cap \alt_4 
 = W_1$ is cyclic of order three, generated by  $\begin{bmatrix}
        0 & 1 \\ 1 & 1 \\
    \end{bmatrix}$.  
In coordinates, this generator of $W_1$ acts on $H^*(V_2) \cong \Z/2[x_1, x_2]$ by sending $x_1 \mapsto x_2$ and $x_2 \mapsto x_1 + x_2$.

The following is Theorem III.1.3 of \cite{AdMi94}.

\begin{theorem}\label{admi}
The invariants of $\Z/2[x_1, x_2]$ under the action of  $W_1$  are generated by a degree two class $$b = {x_1}^2 + x_1x_2 + {x_2}^2$$ and the degree three classes 
$$c_+ = {x_1}^3+{x_1}^2x_2+{x_2}^3 \;\; {\rm and} \;\; c_- = {x_1}^3+x_1{x_2}^2+{x_2}^3.$$  
\end{theorem}

To understand the action of $W_2$ we start with
$W_1 \times W_1 \subset W_2$
acting on the two sets of variables as above. 
Moreover, $\si_2 \subset W_2$ acts on $\Z/2[x_1, x_2, y_1, y_2]$ by sending $x_i \leftrightarrow y_i$. 
There is a third type of element  of $W_2$, or
more generally $W_n$, which together with the first two types of actions generates the full group.

\begin{definition}
    Let $\tau$ denote  $\begin{bmatrix}
        0 & 1 \\ 1 & 0 \\
    \end{bmatrix} \in GL_2( \F_2)$. Define $\tau_{j,k}$ as 
    the element  $(m_1, m_2, \cdots, m_n) \in  GL_2( \F_2)^{\times n} \subset
    W_n^{sym}$ for which
    $m_j = m_k = \tau$ and all other $m_i$ are the identity.  
\end{definition}

Because it constitutes an even permutation, $\tau_{j,k} \in W_n$.   In fact, the $\tau_{j,k}$ generate the quotient
of $W_n$ by $\si_n \wr W_1$. 
The element $\tau_{j,k}$ acts on $H^*(E_{alt}) \cong  \Z/2[x_1, x_2]^{\otimes n}$ by exchanging 
$x_1$ and $x_2$ on both the $j$th and $k$th tensor factors.

Before treating general results, we calculate some invariants for $W_2$.  
We begin with the invariants under the $W_1 \times W_1$ action and (minimally) symmetrize to produce $W_2$ invariants.  
For example, in degree two, we start with $ x_1^2 + x_1x_2+x_2^2$
and symmetrize with respect to the $S_2$ action  exchanging $x$ and $y$ to get
\[ {x_1}^2+x_1x_2+{x_2}^2 + {y_1}^2+y_1y_2+{y_2}^2, \]
which is also invariant under the action of $\tau_{1,2}$ 
and thus is invariant under $W_2$.
In degree three, we begin with $ x_1^3 + x_1^2x_2 + x_2^3$
and first symmetrize with respect to the $\si_2$ action
 to obtain
\[  x_1^3 + x_1^2x_2 + x_2^3 + y_1^3 + y_1^2y_2+ y_2^3.\]
Then, we further symmetrize by the action of $\tau_{1,2}$, getting
\begin{multline*}
    x_1^3 + x_1^2x_2 + x_2^3 + y_1^3 + y_1^2y_2+ y_2^3 + x_1^3 + x_1x_2^2 + x_2^3 + y_1^3 + y_1y_2^2+ y_2^3 \\
    = {x_1}^2x_2+x_1{x_2}^2+{y_1}^2y_2+y_1{y_2}^2.
\end{multline*}
This is the same polynomial we arrive at by starting instead with $x_1^3+x_1x_2^2 + x_2^3$ and minimally symmetrizing. Thus, the two different degree three generators we had for $W_1$ invariants ``merge" as we
use them to produce $W_2$ invariants. 

For illustration,
we list some elements of  $\Z/2[x_1, x_2, y_1, y_2]^{W_2}$, along with their 
expression under the  isomorphism of $\Z/2[x_1, x_2, y_1, y_2]$
with $ \Z/2[x_1, x_2]^{\otimes 2}$ and the notation of Theorem~\ref{admi}:

\begin{itemize}
    \item ${x_1}^2+x_1x_2+{x_2}^2 + {y_1}^2+y_1y_2+{y_2}^2$, which corresponds to $b \otimes 1 + 1 \otimes b$;
    \item ${x_1}^2x_2+x_1{x_2}^2+{y_1}^2y_2+y_1{y_2}^2$, which corresponds to $(c_+ + c_-) \otimes 1 + 1 \otimes (c_+ + c_-)$;
    \item $({x_1}^2+x_1x_2+{x_2}^2)({y_1}^2+y_1y_2+{y_2}^2)$ which corresponds to $b \otimes b$;
    \item $({x_1}^2+x_1x_2+{x_2}^2)({y_1}^2y_2+y_1{y_2}^2)+({y_1}^2+y_1y_2+{y_2}^2)({x_1}^2x_2+x_1{x_2}^2)$, 
    which corresponds to $b \otimes (c_+ + c_-) + (c_+ + c_-) \otimes b$;
    \item $({x_1}^3+{x_1}^2x_2+{x_2}^3)({y_1}^3+y_1{y_2}^2+{y_2}^3)+({y_1}^3+{y_1}^2y_2+{y_2}^3)({x_1}^3+x_1{x_2}^2+{x_2}^3)$,
    which corresponds to $c_+ \otimes c_- + c_- \otimes c_+$;
    \item $({x_1}^3+{x_1}^2x_2+{x_2}^3)({y_1}^3+{y_1}^2y_2+{y_2}^3)+({y_1}^3+y_1{y_2}^2+{y_2}^3)({x_1}^3+x_1{x_2}^2+{x_2}^3)$,
    which corresponds to  $c_+ \otimes c_+ + c_- \otimes c_-$.
    
\end{itemize}

Because $c_+ \sim c_-$ under the exchange relations, the classes in degrees 3, 5, and 6 are all negligible.  In fact we prove below that only 
$b \otimes 1 + 1 \otimes b$ and $b \otimes b$ are needed as multiplicative
generators in the essential quotient.  

For general $W_n$, first observe that 
$$\left( \Z/2[b]^{\otimes n}\right)^{\si_n} \subseteq H^*(E_{alt})^{W_n} \subseteq \Z/2[b, c_+, c_-]^{\otimes n},$$
the first inclusion being an immediate check of invariance
and the second following from the fact that  $W_1^{\times n} \subset W_n$ and Theorem~\ref{admi}.  The key calculation 
in determining the essential quotient of the cohomology
of $\alt_{4n}$ is the following.

\begin{theorem}\label{Wninvts}
    For $n>1$, the quotient of $H^*(E_{alt})^{W_n}$, 
    by its intersection with the negligible ideal
in $H^*(E_{alt})$  is spanned
    by the image of the subring $\left( \Z/2[b]^{\otimes n}\right)^{\si_n}$ of $ H^*(E_{alt})^{W_n}$.
\end{theorem}

\begin{proof}
    We first treat the $n=2$ case.  Let $H \subset W_2$ be the subgroup generated by $\tau_{1,2}$ and the permutation of the two factors of $V_2$.   
    The action of $W_2$, and thus that of $H$, permutes the basis of tensor products of monomials.  We may therefore decompose a $W_2$-invariant $x$ as a sum 
    $$x = \sum_i \sum_{(\sigma) \in H/K_i} \sigma (f_i \otimes g_i),$$
    where $f_i$ and $g_i$ are monomials and 
    each $K_i$ is a subgroup of $H$.
      Explicitly, considering possible subgroups $K_i \subset H \cong C_2 \times C_2$, each sum
      $\sum_{(\sigma) \in H/K_i} \sigma (f_i \otimes g_i)$ is one of the following, where $f$, $f_+$, $f_-$, $g$, $g_+$ and $g_-$ are all monomials:
\begin{enumerate}
    \item $f \otimes f$ fixed by $\tau$;
    \item $f_+ \otimes f_+ + f_- \otimes f_-$, where $\tau(f_+) = f_-$ and $f_+ \neq f_-$;
    \item $f_+ \otimes f_- + f_- \otimes f_+$, where $\tau(f_+) = f_-$ and $f_+ \neq f_-$;
    \item $f \otimes g + g \otimes f$, where $f$ and $g$ are fixed by $\tau$;
    \item $f_+ \otimes g_+ + g_+ \otimes f_+ + f_- \otimes g_- 
    + g_- \otimes f_-$, where  $\tau(f_+) = f_-$, $f_+ \neq f_-$,$\tau(g_+) = g_-$ and $g_+ \neq g_-$.
\end{enumerate}
The invariants of $\Z/2[b,c_+,c_-]$ under $\tau$ are $\Z/2[b,c]$, where $c = c_+ + c_-$.  But  $c_+ +  c_- = {x_1}^2 x_2 + x_1 {x_2}^2$, which is negligible.  Thus any $\tau$ invariant is equivalent to an element of $\Z/2[b]$.  
Also, because $c_+ \sim c_-$  any monomial $f_+ \in \Z/2[b, c_+, c_-]$ will be equivalent to $\tau(f_+)$ modulo the negligible ideal.  Considering the orbit types above in the quotient by the negligible ideal we deduce
the following equivalences:
\begin{enumerate}
    \item $f \otimes f$ is equivalent to an element of $(\Z/2[b]^{\otimes 2})^{\si_2}$;
    \item $f_+ \otimes f_+ + f_- \otimes f_-$ is zero  because $f_+ \sim f_-$;
    \item $f_+ \otimes f_- + f_- \otimes f_+$ is zero;
    \item $f \otimes g + g \otimes f$ is equivalent to an element of $(\Z/2[b]^{\otimes 2})^{\si_2}$;
    \item $f_+ \otimes g_+ + g_+ \otimes f_+ + f_- \otimes g_- 
    + g_- \otimes f_-$ is zero.
\end{enumerate}
Thus, considering the $H$-orbit monomial decomposition of any $W_2$ invariant, it will be equivalent to an element of $(\Z/2[b]^{\otimes 2})^{\si_2}$, establishing the result in this case.

In general, let $H_{i,j}$ be the subgroup of $W_n$ generated by
$\tau_{ij}$ and the permutation of the $i$th and $j$th
factors of $V_2$.
Decompose a $W_n$ invariant into $H_{i,j}$ orbits of tensor products of monomials,
each consisting of a sum of one, two or four tensor products of monomials as in the argument above.
Following the orbit analysis above, 
each such orbit will be equivalent to sum of one or two 
tensor products of monomials
whose $i$th and $j$th entries are in $\Z/2[b]$ and which
is invariant with respect to permuting the $i$th and $j$th entries.
    Applying this argument
    for all $i, j$ we see that modulo
    negligible elements, any $W_n$ invariant is equivalent to
    an element of $\left(\Z/2[b]^{\otimes n}\right)^{\si_n}$.
\end{proof} 

Thus, when $n>1$ the involutions $\tau_{i,j}$ force much
of the cohomology to be negligible. 

We can now complete the proofs of our main results.\\

\begin{proof} 
[{\bf Proof of Theorem~\ref{main}.}]
    We established the first case and half of the second case above.  The third case is immediate from Theorem~\ref{admi} and the fact that $c_+ = c_-$
    modulo the negligible ideal.
    
    The remaining part of the second case states that that essential cohomology
    of $\alt_{4m}$ with $m>1$
    are generated by Stiefel-Whitney classes of even degree.
    Recall that the restriction map $H^*_{ess}(A_{4n}) \to H^*_{ess}(E_{alt})$ is injective. Moreover, the image of this restriction map is contained in
$H^*_{ess}(E_{alt},Z/2)^{W_n}$. Thus it suffices to show that
$Q_n$ is generated by Stiefel-Whitney classes of even degree.

    The standard representation of $\si_4$ restricted to $V_2$ can
    be decomposed as a sum of one-dimensional representations.  The three non-trivial
    summands are spanned by $e_1 + e_2 - e_ 3 - e_4$, by 
    $e_1 - e_2 + e_3 - e_4$ and by $e_1 - e_2 - e_3 + e_4$.
    The first one is fixed by $(12)(34)$ and acted on non-trivially
    by the other two elements of $V_2$.  It is thus the pull-back
    of the sign representation of $C_2$ by the quotient of $V_2$
    by the subgroup given by the identity and $(12)(34)$.
    The second and third representations are also pulled back
    from such quotient homomorphisms, for the other two 
    non-trivial elements.
    The associated bundles have first Stiefel-Whitney classes given by $x_1$, $x_2$
    and $x_1 + x_2 \in H^1(V_2 )$ respectively.  Thus the total
    Stiefel-Whitney class is $(1 + x_1)(1+ x_2)(1 + x_1 + x_2)$,
    which in our notation above is equal to $1 + b + c$.

    The standard representation of $\si_{4n}$  restricts to 
    $n$ copies of these, and thus has total Stiefel-Whitney
    class $(1 + b + c)^{\otimes n} \in \bigotimes_n  \Z/2[x_1, x_2]$.  Modulo  negligible elements this is
    equivalent to $(1+b)^{\otimes n}$.  Thus, elementary symmetric invariants in $1$ and $b$
    are all restrictions of Stiefel-Whitney classes.  By 
    Theorem~\ref{Wninvts}, 
    Stiefel-Whitney classes 
    thus generate the essential quotients of cohomology
    of alternating groups.  

    The general case builds on these and follows the same logic.  Choose $E$ to be the subgroup of $G$ given by the product of $E_{sym}$ and $E_{alt}$
    of its factors, which is detecting.
    The Weyl group is the product of Weyl groups, and
    the module of 
    Weyl group invariants is the tensor product of Weyl group
    invariants of the factors. These are exhausted by the ring generated
    by the pullbacks of generators of the factors.
 \end{proof}

\begin{proof}[{\bf Proofs of Theorems~\ref{relations} and \ref{basis}.}]
We first establish relations and an additive basis when $G = \alt_4$.
  Immediate calculation following Theorem~\ref{admi} gives
$c_+ c_- = {w_2}^3$. Passing to the essential quotient yields
$c^2 = {w_2}^3$. Additively, one then has at most a single
monomial in each degree, either ${w_2}^n$ in even degrees or ${w_2}^{n-1} c$ in odd degrees.  Restricting to any order-two 
subgroup of the detecting subgroup $V_2$ will send $b$ to $t^2$
and $c$ to $t^3$, showing these monomials form a basis. 

To establish relations between products of Stiefel-Whitney classes 
for arbitrary products of symmetric and alternating groups,
it suffices to consider $G= \si_{2n}$.  We follow a method of Serre.
Recall that the detecting subgroup $E_{sym}$ for $\si_{2n}$ is generated
by transpositions $(12)$, $(34)$, etc.
Consider restriction map $\psi_i$ from
the cohomology of $E_{sym}$ to the order-two subgroup
defined by the product of the first $i$ of these transpositions. Together these $\psi_i$ give a homomorphism $\Psi$ from 
$H^*(E ) \cong \Z/2[x_1, \ldots, x_n]$ to the coproduct $\coprod_n \Z/2[t]$ where the $\psi_i$ factor sends $x_j$ to $t$ for $j \le i$ or to $0$ when $j > i$.  Because any order-two subgroup
of $\si_{2n}$ is conjugate to one of these order-two subgroups,
the composite of $\Psi$ with restriction to $E_{sym}$ passes to 
$H^*_{ess}(\si_{2n} )$, and defines
an injection on that quotient.

We  calculate relations among  Stiefel-Whitney classes by looking at their images under restriction to $E_{sym}$, namely to symmetric
polynomials as calculated in the proof of Theorem~\ref{main},
composed with $\Psi$.   
Under this composite map to $\coprod_n  \Z/2[t]$, 
    \[ w_i  \text{ maps to }\left(\binom{1}{i} t^i, \binom{2}{i} t^i, \cdots, \binom{n}{i} t^i, \cdots \right). \]
 The product $w_iw_j$ maps to 
    \[  \left(\binom{1}{i}\binom{1}{j} t^{i+j}, \binom{2}{i}\binom{2}{j} t^{i+j}, \cdots, \binom{n}{i}\binom{n}{j} t^{i+j}, \cdots \right).\]
    If (and only if) $i$ and $j$ are bitwise disjoint,  by Lucas's Theorem working modulo two the above is equal to 
    \[ \left(\binom{1}{i+j} t^{i+j}, \binom{2}{i+j} t^{i+j}, \cdots, \binom{n}{i+j} t^{i+j}, \cdots \right),\]
    which is  the image of $w_{i+j}$.
Through restriction, this relation holds as well in essential cohomology
of alternating groups, and more generally for classes pulled back
to a product of symmetric and alternating groups.

By these relations,
the $w_{2^i}$ can be taken as ring generators, and the square-free
products of these are equal to $w_j$ where $j = \sum 2^i$ is the
dyadic expansion of $j$.
Any descendent of $w_j$ 
will have total degree which differs from $j$ by
the multiples of the $2^i$ which occur, namely 
multiples of the $2$-part
of $j$. 

To see that the descendents of the $w_j$ are independent in $H^*_{ess}(\si_{2n} )$, 
we calculate that 
the image under $\Psi$ of a degree $d$ descendent of $w_j$ will be $t^{d-j} \cdot \Psi(w_j)$.  
But the image of  $w_j$ in $\coprod_n  \Z/2[t]$ will start
with exactly $j-1$ zeros. Thus, collectively in degree $d$, 
the $t^{d-j} \cdot \Psi(w_j)$ are  ``upper-triangular'', so
these descendents are linearly independent.

The argument for linear independence in the essential quotient of an alternating group is similar.  We can choose an
appropriate sequence of order-two 
elements of $E_{alt}$, such
as $(12)(34)$, $(12)(34)(56)(78)$, etc.,  and define 
$\phi_i$ to be the restriction from the cohomology of the
alternating group to the cohomology of the $i$th of the corresponding subgroups.
Recall that $w_{2j}$ maps to the $j$th symmetric polynomial
on classes $b = {x_1}^2 + x_1 x_2 + {x_2}^2$ in $H^*(E_{alt} ) 
\cong \Z/2[x_1, x_2]^{\otimes n}$.  The composite of this
restriction with $\phi_i$  will send 
 \[ w_{2j}  \text{  to }\left(\binom{1}{j} t^{2j}, \binom{2}{j} t^{2j}, \cdots, \binom{n}{j} t^{2j}, \cdots \right). \]
The argument above that descendents of $w_j$ are linearly independent
in $H^*_{ess}(\si_{2n})$ now applies verbatim to descendents 
of $w_{2j}$ in $H^*_{ess}(\alt_{2m} )$.

For products of symmetric and alternating groups, 
we consider products of Stiefel-Whitney classes pulled back from each factor, along
with possible $c$ or $w_2 c$ classes from $\alt_4$. These products and their descendents span since the $w_{2^i}$ and $c$ generate as an exchange algebra.

To see linear independence, we recast the above results in terms of pairings between the span of descendents 
of appropriate monomials ($w_j$ for $G = \si_{2n}$, $w_{2j}$ for $G = \alt_{2m}$, and $w_2$ or $c$ for $G = \alt_4$) and ${\rm Hom}(H^*_{ess}(E ), \Z/2)$. In the preceding arguments 
the process of taking the coefficient of $t^k$ defined
such homomorphisms.  
We used restriction to the cohomology of chosen order-two subgroups, 
followed by the standard non-zero homomorphisms from their cohomology to show that such a pairing is full rank for a single symmetric or alternating group. 
For a product of symmetric and alternating groups,
we take the order-two subgroup defined by all products of order-two elements constructed for the factors.  The pairing with ${\rm Hom}(H^*_{ess}(E), \Z/2 )$ will include
 the tensor product of pairings with the individual groups, and thus will be full rank.
\end{proof}

Naively, there is no K\"unneth theorem for essential cohomology, as
one can exchange classes pulled back from different factors.  But, pull-backs of generators still generate,
and this argument shows that the square-free quotient is  the tensor product of the square-free quotients
for constituent factors.

\bibliographystyle{alpha}
\bibliography{essential}

@preamble{"\def\cftil#1{\ifmmode\setbox7\hbox{$\accent"5E#1$}\else     \setbox7\hbox{\accent"5E#1}\penalty 10000\relax\fi\raise 1\ht7     \hbox{\lower1.15ex\hbox to 1\wd7{\hss\accent"7E\hss}}\penalty 10000     \hskip-1\wd7\penalty 10000\box7} " #
   "\def\cfudot#1{\ifmmode\setbox7\hbox{$\accent"5E#1$}\else     \setbox7\hbox{\accent"5E#1}\penalty 10000\relax\fi\raise 1\ht7     \hbox{\raise.1ex\hbox to 1\wd7{\hss.\hss}}\penalty 10000     \hskip-1\wd7\penalty 10000\box7} "}

@article{GiSi21,
	author = {Giusti, Chad and Sinha, Dev},
	doi = {10.1515/crelle-2020-0016},
	fjournal = {Journal f\"ur die Reine und Angewandte Mathematik. [Crelle's Journal]},
	issn = {0075-4102,1435-5345},
	journal = {J. Reine Angew. Math.},
	mrclass = {20J06 (55P35 55S10)},
	mrnumber = {4227589},
	mrreviewer = {David\ Benson},
	pages = {1--51},
	title = {Mod-two cohomology rings of alternating groups},
	url = {https://doi.org/10.1515/crelle-2020-0016},
	volume = {772},
	year = {2021}}

@article{miln70,
    author = {Milnor, John},
    title = {Algebraic K-theory and quadratic forms}, 
    journal = {Invent. Math.},
    Volume = {9},
    year = {1969/1970}, 
    pages = {318-344}
}

@article {B-FS94,
    AUTHOR = {Bayer-Fluckiger, Eva and Serre, Jean-Pierre},
     TITLE = {Torsions quadratiques et bases normales autoduales},
   JOURNAL = {Amer. J. Math.},
  FJOURNAL = {American Journal of Mathematics},
    VOLUME = {116},
      YEAR = {1994},
    NUMBER = {1},
     PAGES = {1--64},
      ISSN = {0002-9327,1080-6377},
   MRCLASS = {11E72 (11E04 12G05)},
  MRNUMBER = {1262426},
MRREVIEWER = {Juliusz\ Brzezi\'nski},
       DOI = {10.2307/2374981},
       URL = {https://doi.org/10.2307/2374981},
}

@article {Serre26,
    author = {Serre, Jean-Pierre},
    title = {Sur les \'el\'ements n\'eegligeables de la cohomologie mod 2. Partial preprint shared through e-mail.},
    year = {2026}
    }

@incollection{Serre24,
    author = {Serre, Jean-Pierre},
title = {Invariants cohomologiques mod 2 et invariants de {Witt} des groupes altern\'es.},  booktitle = {Jean-Pierre Serre Oeuvres, Volume V},
  publisher = {Springer-Verlag },
  year      = {2026},
  pages     = {645--671},
  address   = {Cham, Switzerland}
}

@article{GM22,
title = {Negligible degree two cohomology of finite groups},
journal = {Journal of Algebra},
volume = {611},
pages = {82-93},
year = {2022},
issn = {0021-8693},
doi = {https://doi.org/10.1016/j.jalgebra.2022.07.039},
url = {https://www.sciencedirect.com/science/article/pii/S0021869322003854},
author = {Matthew Gherman and Alexander Merkurjev},
}

@article{GM24,
title = {Krull dimension of the negligible quotient in mod p cohomology of a finite group},
journal = {Journal of Pure and Applied Algebra},
volume = {228},
number = {3},
pages = {107489},
year = {2024},
issn = {0022-4049},
doi = {https://doi.org/10.1016/j.jpaa.2023.107489},
url = {https://www.sciencedirect.com/science/article/pii/S0022404923001718},
author = {M. Gherman and A. Merkurjev},
}

@phdthesis{Gherman23,
  title={Negligible Cohomology},
  author={Gherman, Matthew Michael},
  year={2023},
  school={University of California, Los Angeles},
  url={https://escholarship.org/uc/item/0zj2b9nk}
}

@book{MaMi79,
	author = {Madsen, Ib and Milgram, R. James},
	isbn = {0-691-08225-1},
	mrclass = {57N70 (55P47 55R35 57Q60)},
	mrnumber = {548575},
	mrreviewer = {Yu. P. Solov\^a${}^2$ev},
	pages = {xii+279},
	publisher = {Princeton University Press, Princeton, N.J.; University of Tokyo Press, Tokyo},
	series = {Annals of Mathematics Studies},
	title = {The classifying spaces for surgery and cobordism of manifolds},
	volume = {92},
	year = {1979}}

@book{AdMi94,
	author = {Adem, Alejandro and Milgram, R. James},
	doi = {10.1007/978-3-662-06280-7},
	edition = {Second},
	isbn = {3-540-20283-8},
	mrclass = {20J06 (18G99 55R35 55R40)},
	mrnumber = {2035696},
	pages = {viii+324},
	publisher = {Springer-Verlag, Berlin},
	series = {Grundlehren der Mathematischen Wissenschaften [Fundamental Principles of Mathematical Sciences]},
	title = {Cohomology of finite groups},
	url = {http://dx.doi.org/10.1007/978-3-662-06280-7},
	volume = {309},
	year = {2004}}

@article{GSS12,
	author = {Giusti, Chad and Salvatore, Paolo and Sinha, Dev},
	doi = {10.1112/jtopol/jtr031},
	fjournal = {Journal of Topology},
	issn = {1753-8416},
	journal = {J. Topol.},
	mrclass = {20J06 (20B30)},
	mrnumber = {2897052},
	mrreviewer = {Nobuaki Yagita},
	number = {1},
	pages = {169--198},
	title = {The mod-2 cohomology rings of symmetric groups},
	url = {http://dx.doi.org/10.1112/jtopol/jtr031},
	volume = {5},
	year = {2012}}

@book {Se03,
    AUTHOR = {Serre,
              Jean-Pierre},
     TITLE = {Cohomological invariants in {G}alois cohomology},
    SERIES = {University Lecture Series},
    VOLUME = {28},
 PUBLISHER = {American Mathematical Society, Providence, RI},
      YEAR = {2003},
     PAGES = {viii+168},
      ISBN = {0-8218-3287-5},
   MRCLASS = {11E72 (12G05)},
  MRNUMBER = {1999383},
MRREVIEWER = {Gr\'egory\ Berhuy},
       DOI = {10.1090/ulect/028},
       URL = {https://doi.org/10.1090/ulect/028},
}

@article{Fesh02,
	author = {Feshbach, Mark},
	coden = {TPLGAF},
	fjournal = {Topology. An International Journal of Mathematics},
	issn = {0040-9383},
	journal = {Topology},
	mrclass = {20J06 (13A50)},
	mrnumber = {MR1871241 (2002h:20074)},
	mrreviewer = {Stuart Martin},
	number = {1},
	pages = {57--84},
	title = {The mod 2 cohomology rings of the symmetric groups and invariants},
	volume = {41},
	year = {2002}}

\end{document}